\documentclass[12pt]{amsart}
\usepackage{amsfonts, amssymb, amsmath, amsthm, euscript}
\usepackage{geometry}
\usepackage{graphicx}
\usepackage{epstopdf}
\usepackage{url}
\usepackage[all]{xy}

\swapnumbers

\theoremstyle{plain}
\newtheorem{thm}{Theorem}[section]
\newtheorem{lem}[thm]{Lemma}
\newtheorem{cor}[thm]{Corollary}

\theoremstyle{definition}
\newtheorem{deff}[thm]{Definition}

\newtheorem{examp}[thm]{Example}
\newtheorem{partition}[thm]{Partition}
\newtheorem{setup}[thm]{Setup}
\newtheorem{rem}[thm]{Remark}

\theoremstyle{remark}

\newtheorem{note}[thm]{}

\def\ds{\displaystyle}

\def\Z{\mathbb{Z}}
\def\RCG{R_{\mathcal{C}}(G)}
\def\C{\mathcal{C}}
\def\D{\mathcal{D}}
\def\A{\mathcal{A}}
\def\x{\textbf{x}}

\begin{document}
\thispagestyle{empty}

\title{ Rings associated to coverings of finite $p$-groups}

\thanks{Research of the second author supported by the Louisiana  BoR [LEQSF(2012-15)-RD-A-20].}

\email{gary.walls@selu.edu, lhwang@pitt.edu}

\keywords{finite $p$-groups, covers of groups, rings of functions}
\it
\footnote { Correspondence: gary.walls@selu.edu \\
                   2010 Mathematics Subject Classification 16S60 \\
                      \it Key words and phrases: \rm  finite $p$-group, covers of groups, rings of functions} \rm
\footnote {Research of the second author supported by the Louisiana BoR [LEQSF(2012-15)-RD-A-20]}
 \begin{center} \bf  Rings associated to coverings of finite $p$-groups \\ \rm

\bigskip
                   Gary WALLS$^1$, Linhong WANG$^2$ \\
 $^1$Department of Mathematics, Southeastern Louisiana University, Hammond, LA 70402 USA \\
$^2$Department of Mathematics, University of Pittsburgh, Pittsburgh, PA \\ 15260 USA
\end{center}
\address{Department of Mathematics\\ Southeastern Louisiana University\\ Hammond, LA 70402\\
                 Department of Mathematics\\ University of Pittsburgh\\ Pittsburgh, PA 15260}

\bigskip
\bf Abstract. \rm  In general the endomorphisms of a non-abelian group do not form a ring under the operations of addition and composition of functions. Several papers have dealt with the ring of functions defined on a group which are endomorphisms when restricted to the elements of a cover of the group by abelian subgroups. We give an algorithm which allows us to determine the elements of the ring of functions of a finite $p$-group which arises in this manner when the elements of the cover are required to be either cyclic or  elementary abelian of rank $2$. This enables us to determine the actual structure of such a ring as a subdirect product.
\par
A key part of the argument is the construction of a graph whose vertices are the subgroups of order $p$ and whose edges are determined by the covering.


\section{Introduction}

Covers of groups by subgroups and rings of functions that act as endomorphisms on each subgroups are studied in many papers including \cite{ckmn, camax,krmax,mks}.

\begin{deff} Suppose that $G$ is a group and $\mathcal{C}$ is a collection of subgroups of $G$. We say that $\mathcal{C}$
is a cover of $G$ provided $\bigcup_{C\in\mathcal{C}} C=G$.
\end{deff}

If all the elements of $\mathcal{C}$ have a certain property $\gamma$, we say that $\mathcal{C}$ is a $\gamma$-covering of $G$. It is well known, e.g. \cite{jaco}, that the endomorphisms of a non-abelian group $G$ do not necessarily form a ring under the operations of function addition and composition. Coverings by abelian subgroups are used to obtain rings of functions on $G$.

\begin{deff} Let $G$ be a group and $\mathcal{C}$ be an abelian-covering of $G$. Define
\[R_{\mathcal{C}}(G) = \{ f:G \rightarrow G \ | \ \text{ for each } C \in \mathcal{C}, f|_C \in \text{ End}(C) \}.\]
\end{deff}
Note that $R_{\mathcal{C}}(G)$ does form a ring under the natural operations on functions, since functions in $R_{\mathcal{C}}(G)$ are endomorphisms when restricted to the subgroups of the cover $\mathcal{C}$. The rings $R_{\mathcal{C}}(G)$ are used in \cite{mks} to classify the maximal subrings of the nearring $M_0(G)$ of the zero preserving functions defined on $G$.
\vspace {.1in}

Let $p$ be a prime and $G$ be a finite $p$-group. In this paper we consider the particular case $(*)$ where all the subgroups in $\mathcal{C}$ are either \emph{maximal cyclic} $p$-groups of $G$ or are \emph{elementary abelian of order} $p^2$. Let $\mathcal{C}$ be a $*$-covering of a finite $p$-group $G$. We prove the following structure theorem to describe the rings arising as $R_{\mathcal{C}}(G)$'s.
\vspace{.1in}

\noindent\textbf{Theorem}
\emph{
Let $G$ be a finite $p$-group, $\mathcal{C}$ be a $*$-covering of $G$. Then $R_{\mathcal{C}}(G)$ is isomorphic to a direct product of rings isomorphic to $M_2(\mathbb{Z}_p)$ or rings of the form of \ref{MatrixRingsAll}.
}
\vspace{.1in}

A key part of our approach is a graph defined in \ref{graph}. The vertices of the graph are the subgroups of $G$ of order $p$ and the edges are determined by the particular covering used.
Each function in $R_{\mathcal{C}}(G)$ is defined on the cyclic subgroups of $G$.
This definition is determined using a specific matrix and associated vector of tuples, even though $f$ may not be linear.
In \ref{ex1}-\ref{ex3}, a few examples are provided to illustrate the theorem. We use the structure theorem to determine conditions for rings arising as $R_{\mathcal{C}}(G)$'s to be of special types. In particular, when the rings are simple we see that the ring $R_{\mathcal{C}}(G)$ must be isomorphic to either $\mathbb{Z} _p$ or $M_2(\mathbb{Z} _p)$. A similar result using a different technique appears in \cite{ckmn}, where covers by subgroups of order $p^2$ are used for finite $p$-groups of exponent $p$.
\vspace{.1in}

Throughout this paper, we always assume that $G$ is a finite $p$-group and $\mathcal{C}$ is a $*$-covering of $G$. We refer to the subgroups in $\mathcal{C}$ as elements of  $\mathcal{C}$ or cells in  $\mathcal{C}$.

\section{Some particular rings}
In this section we present some particular rings needed in order to state the conclusion of our main result. For any positive integer $n$, the endomorphism ring $\text{End}(\mathbb{Z}_{p^n})$ is ring-isomorphic to $\mathbb{Z}_{p^n}$, and is a simple ring if and only if $n=1$. Further, $\text{End}(\mathbb{Z}_p \times \mathbb{Z}_p)$ is isomorphic to $M_2(\mathbb{Z}_p)$, the ring of $2 \times 2$ matrices over $\mathbb{Z}_p$, and so is always simple.
In addition, we will need a subdirect product ring in \ref{MatrixRingsAll}, which is developed by first constructing the matrix ring $N_{i_1,\, i_2,\, \ldots,\, i_n}^{m+n}$ in \ref{MatrixRingsOrderp} and the ring $R_{\Lambda(K)}$ in \ref{MatrixRingsCyclic}.

\begin{note}\label{MatrixRingsOrderp}
Given integers $m> 0$ and $n\ge 0$, we define a ring of $(m+n) \times (m+n)$-
matrices as follows
\[ N_{i_1,\, i_2,\, \ldots,\, i_n}^{m+n}=
\left\{
\left[
\begin{array}{c|c}
\lambda I_m & J(\nu_1, \cdots, \nu_n) \\
\hline
0  & D(\mu_1, \cdots, \mu_n)
\end{array}
\right]
| \
\lambda, \nu_1,\, \ldots,\, \nu_n, \mu_1,\, \ldots,\, \mu_n \in \mathbb{Z}_p
\right\}
\]
where \[D(\mu_1, \cdots, \mu_n)=
\left[ \begin{array}{ccc}
\mu_1 &   \cdots   & 0 \\
\vdots   & \ddots & \vdots \\
0 &     \cdots   & \mu_n
\end{array} \right]
\]
and $J(\nu_1, \cdots, \nu_n)$ is an $m \times n$ matrix which has value $\nu_j$ at the $(i_j, j)$ entry, for $1\leq j \leq n$ and $1 \le i_1,\, \ldots,\, i_n \le m$, and zeroes elsewhere. Note
that there is at most one nonzero entry in each column of $J(\nu_1,\, \cdots, \nu_n)$. It is easy to see that $N_{i_1,\,i_2,\, \ldots,\, i_n}^{m+n}$ is a ring and that
\[I_m^{m+n}=
\left\{
\left[
\begin{array}{c|c}
\lambda I_m & 0 \\
\hline
0 & 0  \end{array} \right]
| \  \lambda \in \mathbb{Z}_p
\right\}
\]
is an ideal of $N_{i_1,\, i_2, \ldots,\, i_n}^{m+n}$, and so $N_{i_1,\,i_2,\, \ldots,\, i_n}^{m+n}$ is never a simple ring, if $n>0$.
\end{note}

\begin{note}\label{MatrixRingsCyclic}
Let $K$ be a subgroup of order $p$ in $G$. There are maximal cyclic subgroups of $G$ that contain $K$.
Assuming there is more than one, let $\Lambda(K)$ be the directed downward lattice of these cyclic subgroups containing $K$.
Define $S(K)$ to be the set of functions on $\Lambda(K)$ that are endomorphisms when restricted to the vertices of $\Lambda(K)$.
On each maximal subgroup in $\Lambda(K)$ of order $p^n$, these functions are multiplications by elements of $\Z_{p^n}$.
As we move down from one vertex to the one below, these functions are multiplied by $p$.
Obviously, these functions will agree on the vertices of $\Lambda(K)$.
Each function in $S(K)$ can be defined by beginning with an element $\lambda \in \mathbb{Z}_p$, which determines an endomorphism on $K$, and then pulling it back up the vertices in $\Lambda(K)$.
Thus, for a fixed $\lambda \in \mathbb{Z}_p$, such a function can be represented as an appropriate tuple $\x=(\lambda_1, \cdots, \lambda_{\phi(K)})$ where $\phi(K)$ is the number of the maximal cyclic subgroups in $\Lambda(K)$ and where each entry $\lambda_i\in \mathbb{Z}_{p^{n_i}}$ determines the endomorphism on a maximal cyclic subgroup of order $p^{n_i}$ in $\Lambda(K)$ such that the properties discussed above hold.
The set of these functions, associated with $K$ and $\lambda$, is denoted as $R_{\Lambda(K),\,\lambda}$.
By this notation, we allow the trivial case when $\Lambda(K)$ is a singleton and $R_{\Lambda(K),\,\lambda}$ is the same as $\{(\lambda)\}$. For each subgroup $K$ of $G$ of order $p$ contained in a maximal cyclic subgroup, the set $R_{\Lambda(K)}=\{R_{\Lambda(K),\,\lambda}\; \text{for}\; \lambda \in \mathbb{Z}_p\}$ does form a ring.
\end{note}

\begin{note}\label{MatrixRingsAll}
A subdirect product of rings can be formed from rings discussed in \ref{MatrixRingsOrderp} and \ref{MatrixRingsCyclic}.
For any matrix in the ring $N_{i_1,\, i_2,\, \ldots,\, i_n}^{m+n}$ of \ref{MatrixRingsOrderp} and certain selected subgroups $K_1, \ldots, K_m$ of $G$ of order $p$, we associate to the diagonal entries $\lambda, \ldots, \lambda, \mu_1, \ldots, \mu_n$ some tuples from $R_{\Lambda(K_i),\,\lambda}$ for $i=1, \ldots, m$ and tuples $(\mu_1), \ldots, (\mu_n)$, respectively. That is
\[
\left(
\left[
\begin{array}{c|c}
\lambda I_m & J(\nu_1, \cdots, \nu_n) \\
\hline
0  & D(\mu_1, \cdots, \mu_n)
\end{array}
\right],
\begin{array}{c}
\x_1\\
\vdots\\
\x_{m}\\
(\mu_1)\\
\vdots\\
(\mu_n)
\end{array}
\right),
\]
where each $\x_i\in R_{\Lambda(K_i), \lambda}$. The arrays constructed in this way form a \emph{subdirect} product of rings $N_{i_1,\, i_2,\, \ldots,\, i_n}^{m+n}$ and $R_{\Lambda(K_1)}, \ldots,R_{\Lambda(K_{m})}$. In particular, if $m=1$ and $n=0$, the subdirect product is isomorphic to a ring of the form of $R_{\Lambda(K)}$, which is isomorphic to a direct product of $\Z_{p^n}$ for various integers $n$.
\end{note}

\section{Determining the elements of the ring $R_{\mathcal{C}}(G)$}

One of the main concerns in determining functions of $R_{\mathcal{C}}(G)$ is to make sure that they are well-defined. We introduce the following graph. The purpose of using such a graph is reflected in Corollary \ref{fonComponents}, which is a direct consequence of Lemma \ref{3-intersecting}. This lemma appeared in \cite{ckmn}. For completeness, we include its proof.
\par
\begin{deff} \label{graph}
Let $T_p(G)$ denote the set of subgroups of $G$ of order $p$. Let $\mathcal{G}$ be the graph whose set of vertices is $T_p(G)$.
Two vertices $A, B $ are joined in $\mathcal{G}$ by an edge provided that there is a cell $C \in \mathcal{C}$ such that $A,\, B \subset C$ and there exist $C_1, C_2, C_3 \in \mathcal{C}$ with intersections $C\cap C_1, C\cap C_2$, and $C\cap C_3$  all distinct subgroups of order $p$.
We call this graph \emph{the $3$-intersecting graph} of $G$. For $A \in T_p(G)$, we let $[A]$ denote the $\mathcal{G}$-\emph{connected component} of $\mathcal{G}$ which contains $A$,
and let \[\text{ Con}(\mathcal{G})= \big\{ [A]\, |\, A \in T_p(G) \big\}\] denote the set of connected components of $\mathcal{G}$.
\end{deff}
\par

\begin{lem}\emph{( cf. \cite[Lemma 6.2]{ckmn})} \label{3-intersecting}
Suppose $A$ and $B$ are two distinct subgroups in $T_p(G)$ connected by an edge in the $3$-intersecting graph $\mathcal{G}$. Then for any $ f \in R_{\mathcal{C}}(G)$ there is a $\lambda \in \mathbb{Z}_p$ such that $f(x)=\lambda x$ for any $x$ in the cell $C=A\times B$.
\end{lem}
\begin{proof} Since $A$ and $B$ are connected by an edge in $\mathcal{G}$, there is a $C \in \mathcal{C}$ so that $C=A\times B \cong \mathbb{Z}_p \times \mathbb{Z}_p$ and there exist $C_1, C_2, C_3 \in \mathcal{C}$ so that $C \cap C_1=\langle e_1\rangle , C \cap C_2=\langle e_2\rangle , \text{ and } C \cap C_3=\langle e_3\rangle $ are three distinct subgroups of order $p$. For any $ f \in R_{\mathcal{C}}(G)$, it is clear that $\langle e_1\rangle , \langle e_2\rangle , \langle e_3\rangle $ must be $f$-invariant. Hence $f(e_i)=\lambda_i e_i$, 
for some $\lambda_1, \lambda_2, \text{ and }
\lambda_3 \in \mathbb{Z}_p$. Note that $C=\langle e_1\rangle \times \langle e_2\rangle $ and so $e_3=\mu_1 e_1 + \mu_2 e_2$ for some nonzero $\mu_1, \mu_2 \in \mathbb{Z}_p$. It follows that $f(e_3)=f(\mu_1 e_1 + \mu_2 e_2)= \mu_1 f(e_1) +\mu_2 f(e_2)=\mu_1 \lambda_1 e_1 + \mu_2 \lambda_2 e_2$. This must equal $\lambda_3 e_3=\lambda_3 (\mu_1 e_1 +\mu_2 e_2)=\lambda_3 \mu_1 e_1 +\lambda_3 \mu_2 e_2$. Since $\{ e_1, e_2 \}$ is a basis of $C$, we get $\lambda_3 \mu_1 = \lambda_1 \mu_1 \text{ and } \lambda_3 \mu_2 = \lambda_2 \mu_2$. It follows that $\lambda_1=\lambda_2$, and so $f|_C$ is scalar multiplication for any $f\in R_{\mathcal{C}}(G)$.
\end{proof}

\begin{cor}\label{fonComponents}
Suppose $A \in T_p(G)$ and $| [A] | >1$. Then $f|_{\cup[A]}$ is multiplication by a scalar $\lambda$ in $\mathbb{Z}_p$.
\end{cor}

\begin{partition} Note that some of the connected components in $\text{Con}(\mathcal{G})$ may be singletons, as cells may not have three distinct intersections. We partition the cells in $\mathcal{C}$ based on their intersections with other cells in $\mathcal{C}$. Set $\ds{\mathcal{C}=\bigsqcup_{i=0}^3 \mathcal{C}_i}$, where
\begin{itemize}
\item[\,] $\mathcal{C}_0=\{ C \in \mathcal{C}\; |\; C \cap C' = \{0\}$ for any $C' \in \mathcal{C}$ and $ C'\neq C\}$
\item[\,] $\mathcal{C}_3=\{ C \in \mathcal{C}\; |\; C_1 \cap C, C_2 \cap C$, and $C_3 \cap C$ are all distinct subgroups of order \hspace*{0.4in} $p$  for some $C_1$, $C_2$, $C_3 \in \mathcal{C}\}$
\item[\,] $\mathcal{C}_2= \{ C \in \mathcal{C} \setminus \mathcal{C}_3\; |\; C_1 \cap C$ and $C_2 \cap C$ are distinct subgroups of order $p$ for \hspace*{0.4in}  some $C_1, C_2 \in \mathcal{C} \}$
\item[\,] $\mathcal{C}_1=\mathcal{C} \setminus (\mathcal{C}_0\cup\mathcal{C}_2\cup\mathcal{C}_3).$
\end{itemize}
Note that a cell $C\in \C_2$ may have more than two cells intersecting with it, for example, $C_1 \cap C$ and $C_2 \cap C=C_3\cap C$ are distinct subgroups of order $p$ for some $C_1, C_2, C_3 \in \mathcal{C}$.  The above partition reveals the structure of a cover $\mathcal{C}$ that we will use to prove the main result.
\end{partition}
\begin{note}\label{f-necessary}
We will show constructively how a function $f\in R_{\mathcal{C}}(G)$ can be defined on $G$ w.r.t. a chosen cover $\mathcal{C}$.
First denote a function  from $\text{ Con}(\mathcal{G})$ to $\mathbb{Z}_p$ by $F$.
Let $x$ be an element of order $p$ in $G$.
If $x$ belongs to a cyclic cell, then $\langle x\rangle$ is $f$-invariant and we define $f(x)=F([\langle x \rangle]) x$.
It is clear that any cyclic cell can only belong to either $\C_0$ or $\C_1$.
If $x$ belongs to a non-cyclic cell, there are several cases.

\emph{Case 1}. If $x\in C$ for some $C\in \mathcal{C}_3$, following Corollary \ref{fonComponents}, we define $f(x)=F([\langle x \rangle])x$.

\emph{Case 2}. If $x\in C$ for some $C\in \mathcal{C}_2$, then there are $C_1, C_2 \in \mathcal{C}$ such that $C \cap C_1 = \langle e_2(C)\rangle$,  $C \cap C_2=\langle e'_2(C)\rangle$ for some element $e_2(C)$ and $e'_2(C)$ of order $p$. Thus $C=\langle e_2(C)\rangle\times \langle e'_2(C)\rangle$ and $x=\alpha e_2(C) + \beta e'_2(C)$ for some $\alpha,\, \beta \in \mathbb{Z}_p$. Note that both $\langle e_2(C)\rangle$ and $\langle e'_2(C)\rangle$ are $f$-invariant. Hence we have
\[f: ( e_2(C),\;  e'_2(C)) \mapsto ( e_2(C),\; e'_2(C))
\left(
\begin{array}{cc}
F([\langle e_2(C)\rangle]) & 0\\
0 & F([\langle e'_2(C)\rangle])
\end{array}
\right)
\]
and $f(x)$ can be defined accordingly. Note that $e_2(C)$ and $e'_2(C)$ are symmetric for each $C\in \C_2$.

\emph{Case 3}. If $x\in C$ for some noncyclic cell $C$ in $\mathcal{C}_1$, then $C \cap C_1 =\langle e_1(C)\rangle$ for some element $e_1(C) \in C$ and $C_1 \in \mathcal{C}$. Pick $b_1(C) \in C$ such that $C=\langle e_1(C)\rangle \times \langle b_1(C)\rangle$.
The choice of $b_1(C)$ is not unique, but $C$ is the unique cell in $\mathcal{C}$ which contains $b_1(C)$. Suppose $x=\alpha e_1(C) + \beta b_1(C)$ for some $\alpha, \, \beta \in \mathbb{Z}_p$. Note that $f(b_1(C))\in C$
and $\langle e_1(C)\rangle$ is $f$-invariant. Hence we have
\[f: ( e_1(C),\;  b_1(C)) \mapsto ( e_1(C),\;  b_1(C))
\left(
\begin{array}{cc}
F([\langle e_1(C)\rangle]) & H(b_1(C)) \\
0 & F([\langle b_1(C)\rangle])
\end{array}
\right)
\]
where $H(b_1(C))$ is a scalar in $\Z_p$.
Then we define
\[f(x)= \Big(\alpha F\big([\langle e_1(C)\rangle ]\big) + \beta H(b_1(C))\Big) e_1(C) + F([\langle b_1(C)\rangle ]) b_1(C).\]
We point out here that if a different choice had been made for $b_1(C)$, it is possible to get the same value for $f(x)$ by choosing a different scalar $H(b_1(C))$.

\emph{Case 4}. If $x\in C$ for some noncyclic cell $C$ in $\mathcal{C}_0$, then $C=\langle b_0(C) \rangle \times \langle b'_0(C)\rangle$ for some $b_0(C),\, b'_0(C) \in C$. The choice of the basis $\{ b_0(C), b'_0(C) \}$ is not unique. Suppose $x=\alpha b_0(C) + \beta b'_0(C)$ for some $\alpha, \, \beta \in \mathbb{Z}_p$. Note that $f(b_0(C))$ and $f(b'_0(C))$ must be in $C$. Hence,
\[f: ( b_0(C),\;  b'_0(C)) \mapsto ( b_0(C),\;  b'_0(C))
\left(
\begin{array}{cc}
F([\langle b_0(C)\rangle]) & B(b'_0(C)) \\
A(b_0(C)) & F([\langle b'_0(C)\rangle])
\end{array}
\right)
\]
where $A(b_0(C))$ and $B(b'_0(C))$ are scalars in $\Z_p$. Then $f(x)$ can be defined accordingly.

After setting the image of $f$ on any element of order $p$ in $G$, now we extend $f$ to the elements of order bigger than $p$,  if there are any. Let $\langle y\rangle$ be a maximal cyclic subgroup of $G$ and $|y|=p^n$ with $n>1$. Note that $\langle y\rangle$ must be a cell in $\mathcal{C}$. Then $f|_{\langle y\rangle} \in \text{ End}(\langle y\rangle) \cong \text{ End}(\mathbb{Z}_{p^n}) \cong \mathbb{Z}_{p^n}$, and so $f(y)=\lambda y$ for some $\lambda \in \mathbb{Z}_{p^n}$. Recall we have already defined $f(y^{p^{n-1}})= F([\langle y^{p^{n-1}} \rangle])y^{p^{n-1}}$, as $x=y^{p^{n-1}}$ is an element of order $p$. Simply working our way up the lattice of the cyclic subgroup $\langle y\rangle$, we can choose a proper scalar $\lambda$ such that $f(y)=\lambda y$ and $f(y^{p^{n-1}})= F([\langle y^{p^{n-1}} \rangle])y^{p^{n-1}}$. Notice that as we work our way up the lattice we have choices, but each choice leads to a different function $f \in R_{\mathcal{C}}(G)$.

It is clear that every function $f$ in $R_{\mathcal{C}}(G)$ arises in the above fashion,
subject to the cover $\C$ (mainly the intersections of the cells in $\C$ such as $e_2(C)$, $e'_2(C)$, $e_1(C)$),
the elements of the form of $b_0(C), \, b'_0(C),\, b_1(C)$ as described above, and the choices of the the function $F$ and scalars $H(b_1(C))$, $A(b_0(C))$, $B(b'_0(C))$.
In terms of notation, $e_i(C)$ for some integer $i$ is always an intersection or contained in an intersection of at least two cells.
Note that, by our notation, it may occur that $\langle e_1(C_1)\rangle=\langle e_2(C_2)\rangle=K$ when $C_1$ is a cell in $\C_1$, $C_2$ is a cell in $\C_2$, and $K=C_1\cap C_2$.
Therefore, to determine $f$, we need a set of elements of order $p$ (indeed subgroups of order $p$) which includes generators of any noncyclic cell in $\C$ and the unique element of order $p$ (indeed the $p$-socle) of any cyclic cell.
\end{note}

\begin{setup}\label{setup}
Given a cover $\C$ of $G$, we setup the following sets of subgroups of order $p$.
The union of these sets is denoted by $\mathcal{B}(\C)$.
\begin{align*}
\mathcal{B}_3(\C) &= \{ \langle g\rangle\, |\,  g \in C \text{ for some } C \in \mathcal{C}_3\; \text{and}\; \langle g \rangle =C\cap C'\; \text{for some}\; C'\in \C\}\\
\mathcal{B}_2(\C) &= \{\langle e_2(C) \rangle,\, \langle e'_2(C) \rangle\, |\,  C \in \mathcal{C}_2\}\\
\mathcal{B}_1^1 (\C)&= \{ \langle e_1(C)\rangle \,|\, C \in \mathcal{C}_1\}\\
\mathcal{B}_1^2 (\C)&= \{ \langle b_1(C)\rangle\, |\, C \in \mathcal{C}_1\; \text{and}\; C=\langle e_1(C)\rangle \times \langle b_1(C)\rangle\; \text{for}\; e_1(C)\in \mathcal{B}_1^1(\C) \}\\
\mathcal{B}_0 (\C)&= \{\langle b_0(C)\rangle,\, \langle b'_0(C)\rangle\, |\, C \in \mathcal{C}_0\}
\end{align*}
As illustrated in \ref{f-necessary}, we also need the following functions.
\begin{align*}
F:\; & \text{ Con}(\mathcal{G}) \rightarrow \mathbb{Z}_p\\
H:\; & \{b_1(C) \, |\, \langle b_1(C)\rangle \in \mathcal{B}_1^2 (\C)  \} \rightarrow \mathbb{Z}_p\\
A:\; & \{b_0(C) \, |\, \langle b_0(C)\rangle \in \mathcal{B}_0(\C)  \}  \rightarrow \mathbb{Z}_p\\
B:\; & \{b'_0(C) \, |\, \langle b'_0(C)\rangle \in \mathcal{B}_0 (\C)  \} \rightarrow \mathbb{Z}_p
\end{align*}
\end{setup}

\begin{thm}\label{MatrixRings}
The ring $R_{\mathcal{C}}(G)$ with a chosen covering $\mathcal{C}$ is isomorphic to a direct product of matrix rings isomorphic to $M_2(\mathbb{Z}_p)$ or rings of the form of \ref{MatrixRingsAll}.
\end{thm}

\begin{proof}
With the $\Z_p$-valued functions and sets of subgroups of order $p$ described in \ref{setup}, any function $f\in \RCG$ can be defined as illustrated in \ref{f-necessary}. To prove the theorem,  we represent the way $f$ is defined on the elements of $\mathcal{B}(\C)$ in terms of a matrix with blocks, which will be denoted by $[f]_{\mathcal{B}(\C)}$.
For this purpose, the subgroups in $\mathcal{B}(\C)$ need to be put in a certain order. The resulting ordered set will be denoted by $\A(\C)$.

We start with $\langle g\rangle \in \mathcal{B}_3(\C)$ if $\mathcal{B}_3(\C)$ is not empty.
Let $D_{\langle g\rangle}$ be the set of the elements $\langle b_1(C)\rangle$
such that $\langle y \rangle \times \langle b_1(C)\rangle$ is a noncyclic cell $C \in \mathcal{C}_1$ for some $\langle y \rangle \in [ \langle g\rangle ]\cap \mathcal{B}_3(\C)$.
Suppose $\langle b_{i1}\rangle, \ldots, \langle b_{i l_i}\rangle$ are from $D_{\langle g \rangle}$ associated to $\langle y_i\rangle\in [ \langle g\rangle ]\cap \mathcal{B}_3(\C)$ for integers $l_i$ and $i=1, \ldots, m$.
The rest of the subgroups $\langle y_{m+1} \rangle, \ldots \langle y_k \rangle \in [ \langle g\rangle ]\cap \mathcal{B}_3(\C)$, if there are any (i.e., $k\geq m$), are either contained in a cyclic cell or in a noncyclic cell $C\in \C_2$.
Let
$\A_g=\{\langle y_1\rangle, \ldots, \langle y_{m}\rangle, \langle z_{11}\rangle, \ldots, \langle z_{1\; l_1}\rangle, \ldots \langle z_{m1}\rangle, \ldots, \langle z_{m l_m}\rangle,\ \langle y_{m+1} \rangle, \ldots \langle y_k \rangle\},$
an ordered set of subgroups from $\mathcal{B}(C)$.

The matrix block of $[f]_{\mathcal{B}(\C)}$ corresponding to $\A_g$ can be determined as shown in \ref{f-necessary} case 1 and case 3.
Set $\lambda= F([\langle y_i\rangle])$ for $i=1, \ldots k$, $\mu_t= F([\langle z_{i j} \rangle])$ and $\nu_t=H(z_{i j})$ for $i=1, \ldots m $ and $t=1, \ldots, n=\sum_{i=1}^m l_i$.
Following the notation in \ref{MatrixRingsOrderp}, we see that the matrix block having the form of
\[
\left[
\begin{array}{cc}
\left[\begin{array}{c|c}
\lambda I_m & J(\nu_1, \cdots, \nu_n) \\
\hline
0  & D(\mu_1, \cdots, \mu_n)
\end{array}\right]
 &0 \\

0 & \lambda I_{k-m}
\end{array}
\right].
\]

It is clear that $\A_g$ is the first part of the ordered set $\A(\C)$. Of course,  before we pursue further, the set $\mathcal{B}(C)$ should be updated by removing the subgroups from the set $\A(\C)=\A_g$.
That is, subsets $\mathcal{B}_3(\C)$, $\mathcal{B}^1_1(\C)$, $\mathcal{B}^2_1(\C)$, and $\mathcal{B}_0(\C)$ of $\mathcal{B}(\C)$ are updated accordingly.
Then we exhaust the set $\mathcal{B}_3(\C)$ by repeating the same process with other subgroups $\langle g'\rangle \in \mathcal{B}_3(\C)$.
Clearly, $[\langle g'\rangle]\neq [\langle g\rangle]$.
Each time, the ordered set $\A(\C)$ is expanded by sets $A_{g'}, \ldots$, while the subsets of $\mathcal{B}(\C)$ are updated accordingly.

Next, we move on to the set $\mathcal{B}_2(\C)$.
If a subgroup $\langle e_2(C) \rangle$ from $\mathcal{B}_2(\C)$ also belongs to other noncyclic cells in $\C_1$, then we add $\langle e_2(C) \rangle$ and all $\langle b_1(C) \rangle$ such that $\langle e_2(C) \rangle \times \langle b_1(C) \rangle$ is a noncyclic cell in $\C_1$.
If not, we simply add $\langle e_2(C) \rangle$. Again, each time $\mathcal{B}_2(\C)$ needs to be updated.
In terms of $[f]_{\mathcal{B}(\C)}$, as shown in \ref{f-necessary} and \ref{MatrixRingsOrderp}, these two cases correspond to a matrix block, with $n$ being zero or positive, having the form of
\[
\left[
\begin{array}{c|c}
\lambda & J(\nu_1, \cdots, \nu_n) \\
\hline
0  & D(\mu_1, \cdots, \mu_n)\\

\end{array}
\right].
\]

Then we continue with subgroups in $\mathcal{B}_1^1$ and $\mathcal{B}_1^2$ in a similar way such that
adding $\langle e_1(C) \rangle$ from $\mathcal{B}^1_1(\C)$ and those $\langle b_1(C) \rangle$ from $\mathcal{B}^2_1(\C)$ for a fixed noncyclic cell $C$
results in a matrix block of the form
\[
\left[
\begin{array}{c|c}
\lambda & J(\nu_1, \cdots, \nu_n) \\
\hline
0  & D(\mu_1, \cdots, \mu_n)\\

\end{array}
\right].
\]

We finish the process by adding the subgroups $\langle b_0(C) \rangle$ and $\langle b'_0(C) \rangle$ from $\mathcal{B}_0$ in pairs or just $\langle b_0(C) \rangle$ if it is from a cyclic cell and not paired with any other subgroup of order $p$ in $\mathcal{B}_0$. Each pair corresponds to a $2\times 2$ block in $M_2(\mathbb{Z}_p)$ as shown in \ref{f-necessary} case 4. Any single subgroup of order $p$ corresponds to a $1\times 1$ block $[\lambda]$ for some $\lambda\in \Z_p$.

To summarize, we have an ordered set $\A(\C)$ of subgroups of $G$ of order $p$,
under which the definition of any function $f\in \RCG$ on elements of $G$ of order $p$ is determined
and represented in terms of a matrix $f_{\mathcal{B}(\C)}$ with blocks as described above, i.e., blocks from $M_2(\Z_p)$ and blocks of the form described in \ref{MatrixRingsOrderp} with choices of $\lambda$'s, $\mu_i$'s and $\nu_i$'s from $\Z_p$.
It is not hard to see that the ordered set $\A(\C)$ may not be unique, but the number and shape of the matrix blocks in $f_{\mathcal{B}(\C)}$ must be fixed, corresponding to the chosen cover $\C$.
Note that any such block matrix with nonzero scalar entries from $\Z_p$ contains enough information to define a function in $\RCG$ on elements of $G$ of order $p$.
The collection of these matrices, with respect to a chosen cover $\C$, does form a ring, which is a direct product of rings isomorphic to $M_2(\mathbb{Z}_p)$ or $N_{i_1,\, i_2,\, \ldots,\, i_n}^{m+n}$ as in \ref{MatrixRingsOrderp}.

To fully represent a function $f\in\RCG$, additional information needs to be attached to the matrix $f_{\mathcal{B}(\C)}$ so that $f$ is defined for elements of order $p^n$ with $n\ge 2$.
We have discussed the definition of $f$ on these elements in the last part of \ref{f-necessary}.
The $p$-socle of these cyclic cells must appear in the list $\A(\C)$.
Following the discussion and notation in \ref{MatrixRingsCyclic}, we associate the diagonal elements of each matrix block of the form, allowing $n=0$,
\[
\left[
\begin{array}{c|c}
\lambda I_m & J(\nu_1, \cdots, \nu_n) \\
\hline
0  & D(\mu_1, \cdots, \mu_n)
\end{array}
\right]
\]
an $(m+n) \times 1$ vector $(\x_1, \ldots,\x_m, \x_{m+1},\ldots,\x_{m+n})^T$, where $\x_i$ are tuples from $R_{\Lambda(K_i),\ \lambda}$ for $i=1, \ldots m$ and for the corresponding subgroup $K_i$ of order $p$ in $\A(\C)$.
If a subgroup $K_i$ does not belong to a cyclic cell of order greater than $p$, then $R_{\Lambda(K_i),\ \lambda}=\{(\lambda)\}$ as pointed in \ref{MatrixRingsCyclic}.
Note that each $x_{m+j}=(\mu_j)$, since the subgroups $\langle b_1(C_i)\rangle$ of order $p$ corresponding to $\mu_1, \cdots, \mu_n$ can only belong to the noncyclic cells $C_i$ as shown in \ref{f-necessary} case 3.
There is no need to associate any vector of tuples to the diagonal of matrix blocks from $M_2(\mathbb{Z}_p)$ because these blocks are corresponding to noncyclic cells in $\C_0$.

Finally, each extended matrix contains enough information to define a function in $\RCG$. Therefore, $R_{\mathcal{C}}(G)$ is isomorphic to a direct product of rings isomorphic to $M_2(\mathbb{Z}_p)$ or rings of the form of \ref{MatrixRingsAll}. The proof is complete.
\end{proof}

\begin{examp}\label{ex1}
Let $G=Q_8=\langle x, y\, |\, x^2=y^2, x^4=1, y^{-1}xy=x^{-1}\rangle$, the quaternion group of order 8. Then $G$ has only one subgroup of order 2.
The only $*$-covering of $G$ is $\C=\{\langle x \rangle, \langle y \rangle, \langle xy \rangle \}$.
Hence
\[
\RCG \cong
\left\{ (a, b, c)\ \Big| \ a, b, c \in \Z_4,\, 2a=2b=2c\right\}.
\]
Note that $|\RCG |=16$.
\end{examp}

\begin{examp}\label{ex2}
Let $G$ be $Q_8\times \Z_2=\langle x, y\rangle \times \langle w \rangle$. Now $G$ has only one non-cyclic subgroup of order 4 and exactly 6 cyclic subgroups of order 4. Consider two $*$-covers
\[\C=\{\langle x\rangle,\ \langle y \rangle,\ \langle xy \rangle,\ \langle xw\rangle,\ \langle yw\rangle,\ \langle xyw\rangle,\ \langle x^2, w\rangle \}\]
and
\[
\D=\{\langle x\rangle,\ \langle y \rangle,\ \langle xy \rangle,\ \langle xw\rangle,\ \langle yw\rangle,\ \langle xy w\rangle,\ \langle w\rangle,\ \langle x^2, w\rangle \}.
\]
Then we have
{\small
\[\RCG=
\left\{
\left(
\left[
\begin{array}{c|c}
\lambda_1 & d \\
\hline
0 & \lambda_2
\end{array}
\right],
\begin{array}{c}
(a,b,c)\\
(\lambda_2)
\end{array}
\right)
\ \Big|\
\lambda_1, \lambda_2, d \in 2\Z_4,\ a, b, c\in \Z_4,\ 2a=2b=2c=\lambda_1
\right\}
\]
}
and
{\small
\[R_{\mathcal{D}}(G)=
\left\{
\left(
\left[
\begin{array}{c|cc}
\lambda_1 & 0 & 0 \\
\hline
0 & \lambda_2 & 0 \\
0 & 0 & \lambda_3
\end{array}
\right],
\begin{array}{c}
(a,b,c)\\
(\lambda_2)\\
(\lambda_3)
\end{array}
\right)
\Big|\
a, b, c\in \Z_4,\ \lambda_1, \lambda_2, \lambda_3 \in 2\Z_4,\ 2a=2b=2c=\lambda_1
\right\}.
\]
}
Notice that $| \RCG |=| R_{\mathcal{D}}(G) | =64$.
\end{examp}

\begin{examp}\label{ex3}
Let $G=D_8\times \Z_2$, where $D_8=\langle x, y\, |\, x^4=1,\ y^2=1,\ (xy)^2=1\rangle$ is the dihedral group of order 8 and $\Z_2=\langle w \rangle$. Take the $*$-cover
\[
\C=\{\langle x\rangle,\ \langle xw \rangle,\ \langle xy, w \rangle,\ \langle w, y\rangle,\ \langle x^2, xy \rangle,\ \langle x^2y, w\rangle,\ \langle x^2w, xy\rangle \}
\]
Then
\[
R_{\mathcal{C}}(G)=
\Big\{
(
\textbf{M},\
\textbf{X}
)
\ |\
a, b\in \Z_4,\ \lambda_1, b_1, d_1, c_1,e_1, h_1, i_1, f_1, g_1 \in 2\Z_4,\ 2a=2b=\lambda_1
\Big\},
\]
where
{\small
$
\textbf{M}=
\left[
\begin{array}{c|cc|ccc}
\lambda_1 & 0  & 0 & 0 & 0 & 0\\
\hline
0 & b_1 & d_1 & 0 & 0 & 0\\
0 & 0 & c_1  & 0 & 0 & 0\\
\hline
0 & 0 & 0 & e_1 & h_1 & i_1\\
0  & 0 & 0 & 0 & f_1 & 0\\
0 & 0 & 0 & 0 & 0 & g_1
\end{array}
\right]
$
and
$
\textbf{X}=
\begin{array}{c}
(a,b)\\
(b_1)\\
(c_1)\\
(e_1)\\
(f_1)\\
(g_1)
\end{array}
$
}
\end{examp}
\

\begin{rem}
If $R_{\mathcal{C}}(G)$ is a simple ring, it follows from Theorem \ref{MatrixRings} that $R_{\mathcal{C}}(G)$ must be isomorphic to either $\mathbb{Z}_p$ or $M_2(\mathbb{Z}_p)$.
A similar result appears in \cite{ckmn}, where covers by subgroups of order $p^2$ are used for finite $p$-groups $G$ of exponent $p$.
An intersection condition on the subgroups in the cover that is equivalent to both $R_{\mathcal{C}}(G)$ being simple and $R_{\mathcal{C}}(G)\cong Z_p$ is developed, see \cite[Theorem 6.10]{ckmn}.
\end{rem}

It is clear that $R_{\mathcal{C}}(G) \cong M_2(\mathbb{Z}_p)$ only occurs when $G \cong \mathbb{Z}_p \times \mathbb{Z}_p$ and $\mathcal{C}=\{ G \}$. Now we derive what it takes for $R_{\mathcal{C}}(G)$ to be isomorphic to $\mathbb{Z}_p$. Suppose that $R_{\mathcal{C}}(G) \cong \mathbb{Z}_p$. It then follows that $G$ must have exponent $p$ and that the $3$-intersecting graph $\mathcal{G}$ of $G$ must be connected, as having more than one connected components leads to non-trivial ideals. Take a nonzero element $a \in G$ and let $C$ be the cell containing $a$. Since $|G| > p^2$, there is an element $b \in G \setminus C$ such that $\langle a\rangle$ and $\langle b \rangle$ are adjacent in $\mathcal{G}$. Hence $|C_G(a)| \ge p^3$. This motivates the following theorem.

\begin{thm} Suppose that $G$ is a finite $p$-group of exponent $p$ and $|C_G(a)| \ge p^3$ for any element $a \in G$. Then there is a $*$-covering $\mathcal{C}$ of $G$ such that $R_{\mathcal{C}}(G) \cong
\mathbb{Z}_p$. Conversely, if $|G| \ge p^3$ and there is $a \in G$ with $|C_G(a)|=p^2$ then $R_{\mathcal{C}}(G)$ is not simple for any $*$-covering $\mathcal{C}$ of $G$.
\end{thm}
\begin{proof} Suppose that $b \in Z(G)$ and $a$ is an element of $G$ with $a \notin \langle b\rangle$. Since $|C_G(a)| \ge p^3$, there is $c_a \in C_G(a) \setminus \langle a,b\rangle$. Consider the cover
\[\mathcal{C} =
\bigcup_{a \in G \setminus \langle b\rangle} \Big\{ \langle a,c_a\rangle, \langle a, b\rangle, \langle b, c_a\rangle, \langle ab, c_a\rangle \Big\}.\]
Now we have $\langle a,b,c_a\rangle \cong \mathbb{Z}_p \times \mathbb{Z}_p \times \mathbb{Z}_p$. It follows that $\langle a,b\rangle \cap \langle a,c_a\rangle =\langle a\rangle, \langle a,b\rangle \cap \langle b, c_a\rangle=\langle b\rangle, \text{ and } \langle a,b\rangle \cap \langle ab, c_a\rangle=\langle ab\rangle$ are three distinct subgroups of order $p$. Hence for all $a \in G \setminus \langle b\rangle$, the subgroups $\langle a\rangle$ and $\langle b\rangle$ are connected by an edge in $\mathcal{G}$. Therefore, $R_{\mathcal{C}}(G) \cong \mathbb{Z}_p$.
\par
On the other hand, if $|G| \ge p^3$ and there is an element $a \in G$ with $|C_G(a)|=p^2$, then $\langle a\rangle$ and
$C_G(a)$ are the only abelian subgroups of $G$ which can contain $a$. It follows that any
$*$-covering of $G$ must contain one or the other. In either case $R_{\mathcal{C}}(G)$ is not simple.
\end{proof}


\end{document}